\documentclass[reqno]{amsart}
\usepackage{amssymb,amsmath,amsthm,amstext,amsfonts}
\usepackage{amsmath,amstext,amsthm,amsfonts}
\usepackage[dvips]{graphicx}
\usepackage{epic,eepic}
\usepackage[dvips]{graphicx}
\usepackage{psfrag}
\usepackage{color}
\usepackage{mathrsfs}

\makeatletter \@addtoreset{equation}{section} \makeatother

\renewcommand\thetable{\thesection.\@arabic\c@table}

\theoremstyle{plain}
\newtheorem{maintheorem}{Theorem}

\newtheorem{theorem}{Theorem }[section]

\newtheorem{corollary}[theorem]{Corollary}

\theoremstyle{definition} \theoremstyle{remark}

\newcommand{\supp}{\operatorname{supp}}

\newcommand{\de} {\delta}       

\newcommand{\vep}{\varepsilon}
\renewcommand{\epsilon}{\varepsilon}

\newcommand{\cU}{\mathcal{U}}
\newcommand{\cS}{\mathcal{S}}

\begin{document}
\large

\title[Trivial and simple spectrum for $SL(2,\mathbb{R})$ cocycles with free base and fiber dynamics]{Trivial 
and simple spectrum for $SL(2,\mathbb{R})$ cocycles with free base and fiber dynamics }
\author{M\'ario Bessa and Paulo Varandas}

\address{M\'ario Bessa, Departamento de Matem\'atica, Universidade da Beira Interior, Rua Marqu\^es d'\'Avila e Bolama,
  6201-001 Covilh\~a,
Portugal.}
\email{bessa@ubi.pt}

\address{Paulo Varandas, Departamento de Matem\'atica, Universidade Federal da Bahia\\
Av. Ademar de Barros s/n, 40170-110 Salvador, Brazil.}
\email{paulo.varandas@ufba.br}

\date{\today}

\maketitle

\begin{abstract}

Let $AC_D(M,SL(2,\mathbb R))$ denote the pairs $(f,A)$ so that $f\in \mathcal A\subset \text{Diff}^{1}(M)$ is a $C^{1}$-Anosov transitive diffeomorphisms and $A$ is an $SL(2,\mathbb R)$ cocycle dominated with respect to $f$. We prove that open and densely 
in $AC_D(M,SL(2,\mathbb R))$ (in appropriate topologies) the pair $(f,A)$ has simple spectrum with 
respect to the unique maximal entropy measure $\mu_f$.
On the other hand, there exists a residual subset $\mathcal{R}\subset \text{Aut}_{Leb}(M)\times L^\infty(M,SL(2,\mathbb R))$, with respect to the separate topology, such that any element $(f,A)$ in $\mathcal{R}$ has trivial spectrum or it is hyperbolic. 
Then, we prove prevalence of trivial spectrum near the dynamical cocycle of an area-preserving map and also for generic cocycles in $\text{Aut}_{Leb}(M)\times L^p(M,SL(2,\mathbb R))$.
\end{abstract}

\section{Introduction}

Let $M$ denote a compact, Riemannian $d$-dimensional manifold  ($d\geq 2$) endowed with a distance $d(\cdot,\cdot)$ inherited from the metric, $Leb$ the volume-measure associated to the volume form on $M$, $\text{Aut}_{Leb}(M)$ the space of the automorphisms of $M$ (i.e. bi-measurable bijections preserving ${Leb}$) and $\text{Homeo}_{Leb}(M)$ the space of homeomorphism 
$f\colon M\rightarrow M$ which leave invariant the measure ${Leb}$, i.e. for all borelians $\mathscr{B}$ we have ${Leb}(\mathscr{B})={Leb}(f^{-1}(\mathscr{B}))$. We endow $\text{Aut}_{Leb}(M)$  with the \emph{weak topology} denoted by $\mathcal{W}$ (i.e. $f_n\rightarrow f$ if and only if ${Leb}(f_n(K)\bigtriangleup f(K))\rightarrow 0$ for all measurable sets $K\subset M$) and endow $\text{Homeo}_{Leb}(M)$ with the $C^0$ topology, given by the metric 
$$d(f, g) := \sup\{d(f(x), g(x)),d(f^{-1}(x), g^{-1}(x)) \colon x \in M\},$$
where $f$ and $g$ are in $\text{Homeo}_{Leb}(M)$.

Let $A\colon M\rightarrow SL(2,\mathbb R)$ be a continuous map when we consider the distance $d$ in $M$ and the uniform operator norm on the  special linear group $SL(2,\mathbb R)$. 
Let $r\in \mathbb N_0$ and $\nu \ge 0$ be such that $r+\nu>0$ and
 let $C^{r+\nu}(M,SL(2,\mathbb R))$  denote the space of $C^r$ cocycles $A\colon M\rightarrow SL(2,\mathbb R)$
 such that $D^r A$ is a $C^{\nu}$-H\"older continuous. Let $L^\infty(M,SL(2,\mathbb R))$ denote the space of all \emph{essentially bounded} maps 
$A\colon M\rightarrow SL(2,\mathbb R)$ endowed with the $L^{\infty}$-norm  defined by 
$$\|A-B\|_{\infty}=\underset{x\in M}{\text{esssup}}\|A(x)-B(x)\| 
	+\underset{x\in M}{\text{esssup}}\|A^{-1}(x)-B^{-1}(x)\|.
	$$  Given $A\in L^\infty(M,SL(2,\mathbb R))$ 
we will denote, by a slight abuse of notation, by \emph{cocycle} the skew-product
\begin{equation*}
\begin{array}{cccc}
F_A \colon  & M\times\mathbb R^{2} & \longrightarrow  & M\times\mathbb R^{2} \\
& (x,v) & \mapsto  & (f(x), A(x)\cdot v),
\end{array}
\end{equation*}
whose joint base and fiber dynamics is given by $F^n_A(x,v)=(f^n(x),A^n(x)\cdot v)$ where 
$$A^{n}(x)=A(f^{n-1}(x)) \circ \dots \circ A(f(x))\circ A(x).$$

In the present paper we intend to study both the dynamical cocycle associated to the derivative of the dynamics as other class of cocycles over a base dynamics $f$ and fiber dynamics $A$. In other words, we ask what is the typical behavior of the deterministic product, determined by $f$, of the product of $SL(2,\mathbb R)$ matrices when we allow changes in both $f$ and $A$.

If $\mu$ is an $f$-invariant probability measure such that $\log\|A^\pm\|\in L^1(\mu)$, then it follows from
Oseledets theorem (see e.g. ~\cite{BP2}) that, for $\mu$-almost every $x$, there exists the \emph{largest
Lyapunov exponent} defined by the limit
$$
\lambda^{+}(f,A,x)=\underset{n\rightarrow{+\infty}}{\lim}\frac{1}{n}\log\|A^{n}(x)\|,
$$
satisfying $\lambda^{+}(f,A,f(x))=\lambda^{+}(f,A,x)$ and it is a non-negative measurable function of $x$. 
Moreover, since we are dealing with elements in $SL(2,\mathbb R)$, for $\mu$-a.e. point $x\in M$ with $\lambda(f,A,x)\not=0$, there is an $A$-invariant splitting of 
the bundle over $x$, $E_{x}^{u}\oplus{E_{x}^{s}}$ which varies measurably with $x$ and such that, if $u\in{E_{x}^{u}}\setminus \{\vec 0\}$ and $s\in{E_{x}^{s}}\setminus \{\vec 0\}$, then

$$
\underset{n\rightarrow{\pm\infty}}{\lim}\frac{1}{n}\log\|A^{n}(x) \cdot u\|=\lambda^{+}(f,A,x)\,\,\,\,\,\, \text{and} \,\,\,\,\,\,\underset{n\rightarrow{\pm\infty}}{\lim}\frac{1}{n}\log\|A^{n}(x)\cdot s\|=\lambda^{-}(f,A,x)=-\lambda^{+}(f,A,x).
$$

If the Lyapunov exponents are all different from each other we say that the spectrum is \emph{simple}. Moreover,
we say that the Lyapunov spectrum is \emph{trivial} if all the Lyapunov exponents are equal, i.e. the Lyapunov spectrum reduces to a single point.

\section{Simple Spectrum}
 
From now on we will consider more regular cocycles with a fiber-bunching assumption, called \emph{domination}, 
over uniformly hyperbolic dynamical systems as we will now describe. We say that  $f\in \text{Diff}^{1}(M)$
is an \emph{Anosov diffeomorphism} if there are constants $C=C_f>0$ and $\theta=\theta_f\in (0,1)$, and a $Df$-invariant splitting $TM=E_f^s\oplus E_f^u$ of the tangent bundle such that $\|Df^n|{E^s}\|\le C \theta^n$
and $\|(Df^n|{E^u})^{-1}\|\le C \theta^n$ for all $n\ge 1$.  We refer to this splitting as an \emph{hyperbolic splitting}. We say that $f$ is \emph{transitive} if it displays a dense orbit, i.e., there exists $x\in M$ such that $\overline{\cup_{n\in\mathbb{Z}}f^n(x)}=M$.
Let us denote by $\mathcal A\subset \text{Diff}^{1}(M)$ the space of $C^{1}$-Anosov diffeomorphisms on $M$ that are transitive, endowed with the $C^1$ Whitney topology, whose norm we denote by $\|\cdot\|_1$.
In fact it is well known that $\mathcal A$ is an open set and that the constants $C>0$ and $\theta\in (0,1)$ can be chosen in a way that for all $g\in\mathcal A$ sufficiently close to $f$ there exists a $Dg$-invariant splitting $TM=E_g^s\oplus E_g^u$ such that $\|Dg^n|{E_g^s}\|\le C \theta^n$ and $\|(Dg^n|{E_g^u})^{-1}\|\le C \theta^n$ for all $n\ge 1$.

We say that a diffeomorphism $f$ is $C^r$-\emph{structurally stable} if there exists a $C^r$-neighborhood of $f$ on which for any element $g$, there exists a homeomorphism $h=h_g$ such that $g\circ h=h\circ f$. Given $\gamma>0$, the homeomorphism $h\colon M\rightarrow M$ is said to be \emph{$\gamma$-H\"older continuous} if there exist $C>0$ such that 
$
d(h(x),h(y))\leq C\,d(x,y)^\gamma
$
for all $x,y\in M$ .
Since Anosov diffeomorphisms are structurally stable, for any $f\in\mathcal A$ there exists $\gamma\in (0,1)$ and a $\gamma$-H\"older continuous homeomorphism $h_g$ close to the identity and such that $g\circ h_g = h_g\circ f$. 
Let $\eta: \mathcal A \to (0,1]$ be a continuous function such that any $f\in\mathcal A$ is $\eta(f)$-H\"older
conjugate to all sufficiently close maps (see e.g.~\cite{KH} for details on regularity of conjugacies).
As a simple consequence, it follows from the structural stability that the set $\mathcal A$ of 
transitive Anosov diffeomorphisms forms an open set in the set of all Anosov diffeomorphisms.

Similarly to before, we can define a hyperbolic cocycle $F_A$ by swapping $Df$ by $A$. In other words, a cocycle
is \emph{hyperbolic} if it admits an $A$-invariant hyperbolic splitting. Hyperbolicity is an open property and it is well-known that the original metric can be changed in order to obtain contant $C$ equal to one. 
We refer the reader to~\cite{KH} for a detailed account on hyperbolicity.   
This put us in a position to recall the notion of domination for cocycles, that roughly means that the cocycle has some partial hyperbolicity in a sense that it behaves like a partial hyperbolic dynamical system whose central direction dynamics is given by the fiber. More precisely, following \cite[Definition~1.1]{BoV04}, given a constant $\nu>0$, 
we say that a $\nu$-H\"older continuous cocycle $A$ is \emph{dominated} for $f$ with hyperbolicity rate $\theta$,  if it satisfies  
$$
\|A(x)\|\,\|A(x)^{-1}\|\, \theta^\nu<1 \text{ for all } x\in M.
$$  
In fact the later is an open condition in the space $C^0(M,SL(2,\mathbb R))$ of continuous cocycles. Therefore, of any $\nu>0$, this is also an open condition in the space $C^\nu(M,SL(2,\mathbb R))$ of $\nu$-H\"older continuous  cocycles endowed with the usual $C^\nu$ norm $\| \cdot \|_\nu= \| \cdot \|_0+| \cdot |_\nu$, where
\begin{equation}
|A|_{\nu}
	=\underset{x\not=y}{\sup}\frac{\|A(x)-A(y)\|}{d(x,y)^\nu}.
\end{equation}

When $\nu=1$ this corresponds to the Lipschitz norm. We denote by $Lip$ the topology induced by the norm $\|\cdot\|_\nu$ with $\nu=1$.

It follows from the thermodynamical formalism for uniformly hyperbolic maps (see e.g.~\cite{Bowen}) that for every $f\in\mathcal{A}$ there exists a unique maximal entropy measure $\mu_f$, and that it has local product structure as we now describe.
 Recall that local stable and local unstable manifolds are $C^1$-embedded submanifolds of $M$ with
 the property that  $W_{\text{loc}}^s(x)$ and $W_{\text{loc}}^u(x)$ vary continuously with $x$ and
there exists $\de>0$ is  all such that for any $x\in M$ and $y,z\in B(x,\delta)$  the intersection 
$$
[y,z]:=W^u_{\text{loc}}(y)\pitchfork W^s_{\text{loc}}(z) \neq \emptyset
$$
 consists of a unique point.
Hence, there exists $N_x^u(\delta)\subset W^u_{\text{loc}}(x)$ a $u$-neighborhood of $x$
and $N_x^s(\delta)\subset W^s_{\text{loc}}(x)$ an $s$-neighborhood of $x$ and a neighborhood 
$N_\de(x)$ of $x$ in $M$ such that the map $\Upsilon_x : N_\de(x) \to  N_x^u(\delta) \times N_x^s(\delta)$ given
by $\Upsilon_x(y) = ( [x,y] , [y,x]  )$ is a homeomorphism. 

An $f$-invariant probability measure $\eta$ has \emph{local product structure} if for any $x\in \supp(\eta)$ (where $\supp(\eta)$ stands for the support of the measure $\eta$) and a small $\de>0$ the measure $\eta\!\mid_{N_x(\de)}$ is equivalent to the  product measure $\eta^u_x \times \eta^s_x$, where $\eta_x^{i}$ denotes the conditional measure of $(\Upsilon_x)_*(\eta\!\mid_{N_x(\de)})$ on  $N^{i}_x(\de)$, for $i\in{u,s}$.   See e.g.~\cite{Bowen} for details. We study Lyapunov exponents with respect to the maximal entropy measure given by
 
 \begin{equation*}
\begin{array}{cccc}
\Lambda \colon  &   AC_D(M,SL(2,\mathbb R)) & \longrightarrow  & [0,+\infty[\\
& (f,A) & \mapsto  & \int_M \lambda^+(f,A,x) \, d\mu_f
\end{array}
\end{equation*}
Let $AC_D(M,SL(2,\mathbb R))\subset \mathcal A \times  C^{1}(M,SL(2,\mathbb R))$ to 
be the set of pairs $(f,A)$ such that $f$ is a transitive Anosov diffeomorphism and $A$ is a $C^1$-cocycle, dominated for $f$. We show that most cocycles in $AC_D(M,SL(2,\mathbb R))$, in the sense of  Theorem~\ref{thm:simple} stated below, have simple spectrum. 
Let us give a brief description of the strategy of the proof. Using that the base dynamics $f$ is Anosov, hence structurally stable, it follows that any diffeomorphism $g$ close to $f$ there exists an H\"older continuous homeomorphism $h_g$ such that $h_g\circ g=f\circ h_g$. Since the maximal entropy measure is preserved by the H\"older continuous conjugation first we can make a reduction and assume that the dynamics $f_0$ is fixed in the Lyapunov exponent function above and one considers cocycles of the form $A\circ h_g$, where $A$ is $C^1$-smooth. The smoothness of the original cocycle plays a key role to guarantee that all close cocycles obtained by conjugation are H\"older continuous with the same regularity.

Let $f\colon M\rightarrow M$ and $g\colon N\rightarrow N$ be invertible measurable maps and measurable $h$-conjugated, say $g\circ h=h\circ f$, for an invertible measurable map $h\colon M\rightarrow N$. The cocycle $A$ over $f$ and the cocycle $B$ over $g$ are \emph{equivalent} if there exists a measurable temperated map $L\colon M\rightarrow SL(2,\mathbb R)$ such that the cohomology equations holds: $A(h^{-1}(x))=L(g(x))^{-1}B(x)L(x)$ for $x\in N$ (cf. Chapter 4 in \cite{BP2}).  We are now in a position to state our first result.

 \begin{maintheorem}\label{thm:simple}
For any $(f,A)\in AC_D(M,SL(2,\mathbb R))$ and $\vep>0$ there exists $(g, B)
 \in AC_D(M,SL(2,\mathbb R))$ such that $\|f-g\|_1<\vep$, $\|A-B\|_{\eta(f)}<\vep$
 and the pair $(g, B)$ has simple spectrum with respect to $\mu_{g}$. 
 Moreover, given $(f,A)\in AC_D(M,SL(2,\mathbb R))$ with simple spectrum with respect to $\mu_f$ there exists an open 
 neighborhood $\cU$ in the $C^1\times Lip$-topology in $AC_D(M,SL(2,\mathbb R))$ such that if 
 $(g,B) \in \mathcal U$, then the pair has simple spectrum with respect  to the maximal entropy measure $\mu_{g}$.
\end{maintheorem}

\begin{proof}
For simplicity reasons, when no confusion is possible, whenever we will say that the pair $(f,A)$  has simple spectrum
we will omit the maximal entropy measure $\mu_f$.

First, let $(f,A)\in AC_D(M,SL(2,\mathbb R))$ and $\vep>0$ be given arbitrary. Using that domination is an open condition, up to consider a smaller value of $\vep$ we may assume without loss of generality that all cocycles 
$C^0$-close to $A$ are dominated with respect to all $g$ that are $C^1$ close to $f$.
Now, since $\mu_f$ has local product structure and the cocycle $A$ is $C^1$ (hence in particular $C^{\eta(f)}$-H\"older continuous) it follows from  \cite{BoV04,Viana} that there exists an open and dense subset $\mathcal S_f$ 
in the set of $C^{\eta(f)}$-H\"older continuous and dominated cocycles for $f$ such that  
$(f,\tilde B)$ has simple spectrum for every $\tilde B\in \cS_f$. 
Since $C^1$-smooth cocycles are $C^{\eta(f)}$-dense in the space of H\"older continuous cocycles, then 
the first assertion will follow by taking $g=f$ and $B\in C^{1}(M, SL(2,\mathbb R))$ in $\cS_f$ with 
$\|A-B\|_{\eta(f)}<\vep$.

Now, we proceed to prove the second assertion in the theorem. 
Assume that the pair $(f,A)\in AC_D(M,SL(2,\mathbb R))$ has simple spectrum. Using again 
\cite{BoV04,Viana} as above, there exists a $C^{\eta(f)}$-open neighborhood $\mathcal V$ of $A$ 
in the space of $C^{\eta(f)}$ cocycles  such that $\mathcal V\subset \mathcal{S}_{f}$. 

In particular, $(f,B)$ has simple spectrum with respect to $\mu_f$ for every $B \in \mathcal V$.
Let $\delta>0$ be such that $B(A,5\delta) \subset \mathcal V$, where $B(A,5\delta)$ stands for the ball, in the $C^{\eta(f)}$-topology centered in $A$ and with radius $5\delta$. 
Let us consider then a $C^1$-neighborhood $\cU_1$ of $f$ so that: 
\begin{itemize}
\item[(i)] for any $g\in \cU_1$ there exists  an $\eta(f)$-H\"older continuous homeomorphism $h_g$ which is close 
to the identity and $g\circ h_g = h_g\circ f$, and 
\item[(ii)] for any $g\in\cU_1$ and $B\in \mathcal V$ the pair $(g,B)$ is dominated.
\end{itemize}
Using that the cocycle $A$ is $C^1$ smooth, then the map $\Xi: \cU_1 \to C^{\eta(f)}(M,SL(2,\mathbb R))$ given by 
$
\Xi: g \mapsto A\circ h_g 
$
is well defined and continuous, where the continuity follows from the fact that the conjugation $g\mapsto h_g$ is continuous. In fact, while the later is a local bijection between a neighborhood of $f$ onto a neighborhood of the identity $id_M$ one can reduce the neighborhood $\cU_1$ of the Anosov diffeomorphism $f$
if necessary so that $\Xi(\cU_1)\subset B(A,\delta)$ and $\|h_g\|_{\eta(f)}<2$ for all $g\in \cU_1$.

Now, for any $g\in \cU_1$ consider the $C^{\eta(f)}$-open set of cocycles $\mathcal V_g=\{ B\circ h_g^{-1} : B \in \mathcal V \}$.  Since there is a unique maximal entropy measure for both $f$ and $g$, then it is preserved by
topological conjugacy and so $(h_g^{-1})_*\mu_g=\mu_f$. Therefore, if $g\in\cU_1$ and $\tilde B=B\circ h_g^{-1}\in \mathcal V_g$, for some $B\in\mathcal V$, then we can use that $f^n= h_g^{-1}\circ g^n \circ h_g$ to check that the
cocycle $\tilde B$ under iteration over $g$ becomes
\begin{align*}
\tilde B^n(y) & = \tilde B(g^{n-1}( y )) \circ \dots \circ \tilde B( g(y) ) \circ \tilde B(y) \\
		& = B ( h_g^{-1}  (g^{n-1}( y ))) \circ \dots \circ  B( h_g^{-1} ( g(y)) ) \circ B(h_g^{-1}(y)) \\
		& = B (  f^{n-1}( h_g^{-1}(y) ) \circ \dots \circ B( f( h_g^{-1}(y))) \circ  B(h_g^{-1}(y)) \\
		& = B^n(h_g^{-1}(y))
\end{align*}
for all $y\in M$, where $B^n(h_g^{-1}(y))$ is obtained by iteration of $B$ over $f$. Thus, the Lyapunov exponents of the cocycle $(g,\tilde B)$ with respect to the maximal entropy 
measure $\mu_g$ are the same as the ones of the cocycle $(f,B)$ with respect to the maximal entropy 
measure $\mu_f$ and, consequently, the pair $(g,\tilde B)$ has simple spectrum.

Finally, we claim that any pair $(g,B)\in AC_D(M, SL(2,\mathbb R))$ such that $g\in\cU_1$ and $\|A-B\|_{1}<\delta$  has
simple spectrum. For that it is enough to show that 
$B\in \mathcal V_g$ or equivalently that $B\circ h_g\in \mathcal V$, since writing $B=(B\circ h_g)\circ h_g^{-1}$.
For any $x,y\in M$ we have
\begin{align*}\label{eq:C1holder}
\|(B\circ h_g-A)(x)-(B\circ h_g-A)(y)\|
		& \le 	 \|(B\circ h_g- A\circ h_g)(x)-(B\circ h_g-A\circ h_g)(y)\| \nonumber \\
		& + \|(A\circ h_g-A)(x)-(A\circ h_g-A)(y)\| \nonumber  \\
		& \le \|A-B\|_{1} \; d(h_g(x),h_g(y))  \\
		& + \|(A\circ h_g-A)\|_{\eta(f)} \; d(x,y)^{\eta(f)}\\
		& \le \delta\|h_g\|_{\eta(f)} \; d(x,y)^{\eta(f)}  + \delta \; d(x,y)^{\eta(f)}.
\end{align*}
Thus, the $|B\circ h_g-A|_{\eta(f)}<3\delta$. Similar estimates provide an upper bound for the $C^0$ norm of the cocycle $B\circ h_g-A$, hence proving that
$B\circ h_g \in B(A,5\delta)\subset \mathcal V$.
This finishes the proof of the theorem.

\end{proof}

We finish this section with some comments. First we mention that the proof of the density above can be taken
in the stronger $C^1 \times C^1$-topology by the results in \cite{BoV04,Viana}.

Secondly, since we use some results from \cite{BoV04,Viana} the previous result also holds from transitive Anosov diffeomorphisms to a basic piece of an Axiom A diffeomorphism. Moreover, we can replace $SL(2,\mathbb R)$ by $SL(d,\mathbb R)$ cocycles for $d\ge 2$ (see~\cite{BoV04}). We opted to state Theorem~\ref{thm:simple} as above for more simple comparison to the characterization of continuous $SL(2,\mathbb R)$ cocycles given in the next theorem.

\section{Trivial and simple spectrum}

In the recent years several results about the generic behavior of $SL(2,\mathbb R)$ cocycles have been proved. We refer to result of Cong ~\cite{C} where it is proved the $L^\infty$ genericity of hyperbolic behavior among \emph{bounded} cocycles when we keep the base dynamics and perturb the fiber dynamics in the $L^\infty$-sense. Moreover, Bochi and Fayad ~\cite{BF} proved that the trivial Lyapunov spectrum prevails in a $C^0$-generic way when we keep the fiber dynamics and allow perturbations in the base dynamics.

In this note we intend to continue the discussion on the problem of knowing what is the prevalent behavior when we consider 
two degrees of freedom, i.e., perturbation of both base and fiber dynamics. In what follows, we endow
the product space $\text{Aut}_{Leb}(M)\times L^\infty(M,SL(2,\mathbb R))$ with the separate topology in $\mathcal{W}\times L^{\infty}$. For full details on the issue of continuity of multivariable real maps see \cite{HW}.

The following result is a very simple consequence of \cite[Proposition 1.7]{BF} which we present in order to contextualize with the object of our study.

\begin{theorem}\label{teo1}(\cite[Proposition 1.7]{BF})
There exists a residual subset $\mathcal{R}\subset \text{Aut}_{Leb}(M)\times L^\infty(M,SL(2,\mathbb R))$ such that any element $(f,A)\in\mathcal{R}$ has trivial spectrum or else the cocycle $(f,A)$ is uniformly hyperbolic.
\end{theorem}

\begin{proof}
We consider the function
\begin{equation*}
\begin{array}{cccc}
\Lambda \colon  & \text{Aut}_{Leb}(M)\times L^\infty(M,SL(2,\mathbb R)) & \longrightarrow  & [0,+\infty[ \\
& (f,A) & \mapsto  & \int_M \lambda^+(f,A,x)\,d Leb(x)
\end{array}
\end{equation*}
Proposition A.2 from ~\cite{BF} assures that, for any fixed $A\in L^\infty(M,SL(2,\mathbb R))$, we have that  
the integrated Lyapunov exponent function $\Lambda(\cdot,A)$ is upper semicontinuous with respect to $\mathcal{W}$. Furthermore, by  \cite[Proposition 2.1]{B}, 
fixing $f\in  \text{Aut}_{Leb}(M)$ we have that  $\Lambda(f,\cdot)$ is upper semicontinuous with respect to the $L^\infty$ norm. Thus, $\Lambda$ is an upper semicontinuous map with respect to the separate topology. As a consequence, for every $\epsilon>0$ there exists a neighborhood $U$ of $(f_0,A_0)$ such that 
$$
\Lambda(f,A)\leq \Lambda(f_0,A_0)+\epsilon, \,\,\,\text{for all}\,\,\,(f,A)\in U.
$$
We claim that if $(f_0,A_0)$ is a continuity point of $\Lambda$, then $\Lambda(f_0,A_0)=0$ or else $(f_0,A_0)$ is uniformly hyperbolic. 
Assume, by contradiction, that $\Lambda(f_0,A_0)>0$ and $(f_0,A_0)$ is not uniformly hyperbolic. Then, there exists a small $\vep>0$ so that 
$\Lambda(f,A)\ge \Lambda(f_0,A_0)/2 > 0$ for all $(f,A)\in \text{Aut}_{Leb}(M)\times L^\infty(M,SL(2,\mathbb R))$ such that $f$ is $\vep$-close to $f_0$ (w.r.t. the topology $\mathcal{W}$) and $\|A-A_0\|_{\infty}<\vep$.
Since we do not have the uniform hyperbolicity property, we are in the conditions of ~\cite{B} (observe that we do not need ergodicity cf. ~\cite[\S 5.3]{B}). Hence, there exists $\tilde A_1\in L^\infty(M,SL(2,\mathbb R))$ such that $\|A_1-A_0\|_{\infty}<\vep$ and moreover $\Lambda( f_0,A_1)=0$ which leads to a contradiction and the theorem is proved.
\end{proof}

We observe that the previous result also holds for $\text{Aut}_{Leb}(M)\times C^0(M,SL(2,\mathbb R))$ instead of $\text{Aut}_{Leb}(M)\times L^\infty(M,SL(2,\mathbb R))$ and the $C^0$ topology instead of the $L^\infty$ one (see e.g. \cite{B}).

\section{Trivial spectrum}

Here, we begin by proving that the trivial spectrum is prevalent for the dynamical cocycle of an area-preserving diffeomorphism near any element of  $\text{Homeo}_{Leb}(M)$. Let $\text{Diff}^r_{Leb}(M)$ stands for the set of $C^r$ area-preserving diffeomorphisms on surfaces. A point $p$ is said to be \emph{periodic} of period $n\in\mathbb{N}$ for  $h\in \text{Homeo}_{Leb}(M)$, if $n$ is the smallest positive integer such that $f^n(p)=p$. We say that a periodic point is \emph{persistent} for $h\in \text{Homeo}_{Leb}(M)$ if for every $\epsilon>0$ there is $\delta>0$ such that if $\hat{h}$ is an homeomorphism $\delta$-$C^0$-close to $h$ then $\hat{h}$ has a periodic point $\hat{p}$ of period $n$ which is $\epsilon$-close to $p$.

\begin{maintheorem}\label{teo1b}
For a $C^0$-dense subset $\mathcal{D}\subset \text{Homeo}_{Leb}(M)$ of $C^2$ diffeomorphisms and $\epsilon>0$ we have that, for $f\in \mathcal{D}$, there is a $C^0$-neighborhood $\mathcal{U}_f$ of $f$ and a residual subset $\mathcal{R}\subset \mathcal{U}_f$ such that for all $h\in \mathcal{R}$: 
$$
\int_M \lambda^+(h,Df,x)d{Leb}(x)<\epsilon.
$$ 
\end{maintheorem}

\begin{proof}
Using ~\cite{N} we know that $C^1$-close to any element in $\text{Diff}^1_{Leb}(M)$ there is an Anosov map or else a map that exhibits dense elliptic periodic orbits. By (\cite{Z}) $\text{Diff}^r_{Leb}(M)$ ($r\geq 2$) is $C^1$-dense in  $\text{Diff}^1_{Leb}(M)$ and by (\cite{Oh}) $\text{Diff}^1_{Leb}(M)$ is $C^0$-dense in  $\text{Homeo}_{Leb}(M)$. 

By performing an arbitrarily small $C^0$-perturbation of an Anosov in  $\text{Diff}^1_{Leb}(M)$ we can obtain an element in $\text{Diff}^1_{Leb}(M)$ $C^1$-far from the $C^1$-open subset of Anosov maps, thus having an element in $\text{Diff}^1_{Leb}(M)$ exhibiting dense elliptic periodic orbits. Therefore, we obtain a $C^0$-dense subset $\mathcal{D}$ of $\text{Homeo}_{Leb}(M)$ with dense elliptic periodic orbits and of class $C^2$. Moreover,  $\mathcal{D}$ can be chosen such that at least one of its (dense) periodic orbits has non-zero second derivative at some point of its orbit. 

In conclusion, we have $f\in \mathcal{D}$ with the dynamical cocycle $Df\colon M\rightarrow SL(2,\mathbb{R})$ of class $C^1$ and, moreover, $f$ has a persistent elliptic periodic point $p = f^n(p)$ such that $D(Df_{f^i(p)})$ is non-zero, for some $i\in\{0,1,...,n-1\}$. 
In fact, the periodic point is persistent since it is an isolated fixed point, for the return map, and thus the Poincar\'e-Lefschetz index is different from 1.

We are in the conditions of ~\cite[Corollary 5]{BF} (see also ~\cite[Theorem 4]{BF}) and so we obtain a neighborhood  $\mathcal{U}_f\subset \text{Homeo}_{Leb}(M)$ of $f$ and a residual subset $\mathcal{R}\subset \mathcal{U}_f$ such that 
$$\int_M \lambda^+(h,Df,x)d{Leb}(x)<\epsilon$$ for all $h\in \mathcal{R}$ and the theorem is proved.

\end{proof}

Now, we also obtain prevalence of trivial spectrum if we consider the $L^p$-measurable $SL(2,\mathbb R)$ cocycles endowed with the $L^p$-norm (cf. ~\cite{AB}). The next result is a direct consequence of Arbieto and Bochi theorem ~\cite{AB}.

\begin{corollary}\label{teo1c}
There exists a residual subset $\mathcal{R}\subset \text{Aut}_{Leb}(M)\times L^p(M,SL(2,\mathbb R))$ such that any element $(f,A)\in\mathcal{R}$ has trivial spectrum.
\end{corollary}

\begin{proof}
The strategy is similar to the proof of Theorem~\ref{teo1}. 
We consider the function endowed with the separate topology, $\mathcal{W}\times L^p$-norm, in its domain:
\begin{equation*}
\begin{array}{cccc}
\Lambda \colon  & \text{Aut}_{Leb}(M)\times L^p(M,SL(2,\mathbb R)) & \longrightarrow  & [0,+\infty[ \\
& (f,A) & \mapsto  & \int_M \lambda^+(f,A,x) \, dLeb(x)
\end{array}
\end{equation*}
By ~\cite[Proposition A.2]{BF} we know that, for any fixed $A\in L^p(M,SL(2,\mathbb R))$, we have that  
the integrated Lyapunov exponent function $\Lambda(\cdot,A)$ is upper semicontinuous with respect to. $\mathcal{W}$. Moreover, by ~\cite[Theorem 2]{AB}, 
for any fixed $f\in  \text{Aut}_{Leb}(M)$ we have that  $\Lambda(f,\cdot)$ is upper semicontinuous with respect to the $L^p$ norm. Hence, since we consider the separate topology it is easy to show that $\Lambda$ is an upper semicontinuous map. Thus, for every $\epsilon>0$ there exists a neighborhood $U$ of $(f_0,A_0)$ such that 
$$
\Lambda(f,A)\leq \Lambda(f_0,A_0)+\epsilon
$$ for all $(f,A)\in U$.
Now we claim that if $(f_0,A_0)$ is a continuity point of $\Lambda$ then $\Lambda(f_0,A_0)=0$. 
Assume, by contradiction, that $\Lambda(f_0,A_0)>0$. Then, there exists a small $\vep>0$ so that 
$\Lambda(f,A)\ge \Lambda(f_0,A_0)/2 > 0$ for all $(f,A)\in \text{Aut}_{Leb}(M)\times L^p(M,SL(2,\mathbb R))$ such that $f$ is $\vep$-close to $f_0$ (with respect to. the topology $\mathcal{W}$) and $\|A-A_0\|_{L^{p}}<\vep$.

However, we can perturb $f_0$ obtaining an ergodic $f_1$ (cf. ~\cite{H}), and, since we are in the conditions of ~\cite[Theorem 1]{BF}, there exists $\tilde A_1\in L^p(M,SL(2,\mathbb R))$ that is $\vep$-$L^p$-close to $A_0$ and such that $\Lambda(\tilde f_1,A_1)=0$ which leads to a contradiction. This proves our claim.
Finally, since the set of continuity points of a semicontinuous function if a residual set the corollary is now proved.

\end{proof}

\section*{Acknowledgements} 
 
MB was partially supported by National Funds through FCT - ``Funda\c{c}\~{a}o para a Ci\^{e}ncia e a Tecnologia", project PEst-OE/MAT/UI0212/2011 and also the project PTDC/MAT/099493/2008. PV was partially supported by
CNPq and FAPESB. The authors would like to thanks Jairo Bochi for suggestions that improved the presentation of the paper.


\end{document}